\documentclass[reqno,11pt]{amsart}
\usepackage{amsmath,amsthm,amssymb}
\usepackage{algpseudocode}
\usepackage{hyperref}
\usepackage[noabbrev,capitalize]{cleveref}
\usepackage{caption} 
\crefname{equation}{}{}
\numberwithin{equation}{section}
\theoremstyle{plain}
\newtheorem{theorem}{Theorem}[section]
\newtheorem{lemma}[theorem]{Lemma}
\newtheorem{corollary}[theorem]{Corollary}
\newtheorem{proposition}[theorem]{Proposition}
\newtheorem{question}[theorem]{Question}
\theoremstyle{definition}
\newtheorem{definition}[theorem]{Definition}
\newtheorem{example}[theorem]{Example}
\newtheorem{remark}[theorem]{Remark}
\usepackage[utf8]{inputenc}
\usepackage[left=1in, right=1 in]{geometry}
\usepackage{url}
\usepackage{enumitem}

\title[New Results on Nyldon Words and Nyldon-like Sets]{New Results on Nyldon Words and Nyldon-like Sets}

\author{Swapnil Garg}
\address{Massachusetts Institute of Technology, Cambridge, MA 02139, USA}
\email{swapnilg@mit.edu}
\date{October 2021}

\begin{document}

\maketitle
\begin{abstract}
Grinberg defined Nyldon words as those words which cannot be factorized into a sequence of lexicographically nondecreasing smaller Nyldon words. He was inspired by Lyndon words, defined the same way except with ``nondecreasing'' replaced by ``nonincreasing.'' Charlier, Philibert, and Stipulanti proved that, like Lyndon words, any word has a unique nondecreasing factorization into Nyldon words. They also show that the Nyldon words form a right Lazard set, and equivalently, a right Hall set. In this paper, we provide a new proof of unique factorization into Nyldon words related to Hall set theory and resolve several questions of Charlier, Philibert, and Stipulanti. In particular, we prove that Nyldon words of a fixed length form a circular code, we prove a result on factorizing powers of words into Nyldon words, and we investigate the Lazard procedure for generating Nyldon words. We show that these results generalize to a new class of Hall sets, of which Nyldon words are an example, that we name ``Nyldon-like sets'', and show how to generate these sets easily.
\end{abstract}

\section{Introduction}
Nyldon words were introduced in 2014 by Darij Grinberg \cite{grinberg}, with the name a play on the related Lyndon words, first studied in the 1950s by Shirshov \cite{shirshov} and Lyndon \cite{lyndon}. While Lyndon words were first defined as those words which are the smallest among their cyclic rotations, the Chen-Fox-Lyndon Theorem states that any word can be written uniquely as a sequence of lexicographically nonincreasing Lyndon words. That is, we can write $w=\ell_1\ell_2\cdots \ell_k$, where $\ell_1 \ge_{\text{lex}} \ell_2 \ge_{\text{lex}} \cdots  \ge_{\text{lex}} \ell_k$ \cite{chen, sirsov}. In a sense, Lyndon words can act as ``primes'' in the factorization of all words. Thus, Lyndon words can also be defined recursively as those words which are either single letters, or which cannot be factorized into a sequence of nonincreasing smaller Lyndon words. By changing the word ``nonincreasing'' to ``nondecreasing'' in this definition, we arrive at Nyldon words, which behave in a surprisingly different way. For example, it is much more difficult to determine whether a word is Nyldon from looking at its cyclic rotations.

In \cite{nyldon}, Charlier, Philibert, and Stipulanti prove an analog of the Chen-Fox-Lyndon Theorem, showing that all words have a unique nondecreasing factorization into Nyldon words. They also give an algorithm for computing the Nyldon factorization of a word, investigate the differences between Nyldon and Lyndon words, and show that Nyldon words form a Hall set; see the end of Section 3 for more information on Hall sets. As unique factorization holds for Nyldon  words, they seem to behave more nicely than another variant of Lyndon words studied recently, the inverse Lyndon words \cite{inverse}.

Lyndon words form a Hall set, with the ordering given by the standard lexicographical ordering. If we instead use the reverse lexicographical ordering to construct a Hall set, we arrive at Nyldon words, showing that these words arise in a quite natural way. As Lyndon introduced his namesake words with the intention of giving bases of free Lie algebras, Nyldon words could also shed light into this area, since they form a Hall set as well.

In this paper, we resolve several questions posed by Charlier, Philibert, and Stipulanti in \cite{nyldon}. In Section 3, we present the algorithm of Melan\c{c}on, which originates in Hall set theory but is arrived at naturally by considering certain factorizations of the free monoid. Specifically, we define a new kind of factorization of the free monoid called a Nyldon-like set. Our definition is a recursive method of generating such sets, and as Nyldon-like sets are shown to be Hall sets, our method is a novel way to generate examples of Hall sets. We demonstrate the power of Melan\c{c}on's algorithm in elucidating certain properties of Nyldon words, Nyldon-like sets, and Hall sets in Section 4. In Section 5, we show that the factorization algorithm conceived in \cite{nyldon} is faster than previously thought. In Section 6, we compute how long the right Lazard procedure takes to generate all the Nyldon words over a given alphabet up to a given length. Finally, in Section 7, we prove a property of Lyndon words using only the recursive definition, answering a question of Charlier et al.\ in \cite{nyldon}.

\section{Background}
Throughout this paper, let $A$ be an alphabet endowed with a total order $<$, with a size at least $2$. We will only work with finite alphabets, and we write a finite alphabet $A$ of size $m$ as $\{0, 1, \dots, m-1\}$ with $0 < 1 < \cdots  < m-1$. The product $uv$ of two words $u, v$ is their concatenation, and the expression $uv^{-1}$ (resp. $v^{-1}u$) is equal to the word $w$ such that $wv=u$ (resp. $vw=u$), so $uv^{-1}$ (resp. $v^{-1}u$) is only defined if $u$ ends with $v$ (resp. begins with $v$). Let $<_\text{lex}$ denote the lexicographical order on words, where $u <_\text{lex} v$ if either $u$ is a proper prefix of $v$, or there exist a word $p$ and letters $i, j \in A$ such that $pi, pj$ are prefixes of $u, v$ respectively, and $i<j$. Let $A^*$ be the set of all finite words (including the empty word $\varepsilon$) over $A$, and let $A^{+}=A^* \backslash \varepsilon$. Let the length of a finite word $w$ be denoted by $|w|$. A word is \textit{primitive} if it is not a power of another word, and a word $w$ is a \textit{conjugate} (or \textit{cyclic rotation}) of $x$ if $w=uv, x=vu$ for words $u, v$.

\begin{definition}\label{def:nyldon} \cite{grinberg, lothaire}
A nonempty word $w$ is Nyldon (resp. Lyndon) if $w=a \in A$ or $w$ cannot be factorized as $(w_1, w_2, \dots, w_k)$ where $w_1, w_2, \dots, w_k$ are Nyldon (resp. Lyndon), $k \ge 2$, and $w_1 \le_{\text{lex}} w_2 \le_{\text{lex}} \dots \le_{\text{lex}} w_k$ (resp. $w_1 \ge_{\text{lex}} w_2 \ge_{\text{lex}} \dots \ge_{\text{lex}} w_k$). Such a factorization is referred to as a Nyldon (resp. Lyndon) factorization (of $w$).
\end{definition}
We provide a table of short binary Nyldon words for reference in \cref{tab:nyldontable}.
\begin{remark}
If $w$ itself is Nyldon (resp. Lyndon), then its Nyldon (resp. Lyndon) factorization is just $(w)$.
\end{remark}

Nyldon words were first extensively studied by Charlier et al., who proved that Nyldon factorization was unique using the following two lemmas:

\begin{lemma}\label{lem:nyldonsuffix} \cite{nyldon}
For a Nyldon word $x$ with a proper Nyldon suffix $s$ (i.e. a suffix $s \neq x$), we have $s <_{\text{lex}} x$.
\end{lemma}

\begin{lemma}\label{lem:longestsuffix} \cite{nyldon}
In the Nyldon factorization of the word $x$ as $(n_1, n_2, \dots, n_k)$, the Nyldon word $n_k$ is the longest Nyldon suffix of $x$.
\end{lemma}

Clearly unique factorization follows from \cref{lem:longestsuffix}. In the following section, we give an alternate, elementary proof of the above two lemmas using only the stated definition of Nyldon words. We also provide an elementary proof of the following theorem about Nyldon words:

\begin{theorem}\label{thm:nyldonprimitive} \cite{nyldon}
Every primitive word has exactly one Nyldon word in its conjugacy class, and no periodic word is Nyldon.
\end{theorem}
\begin{table}[h!]
\centering
\begin{tabular}{ |p{3cm}|p{3cm}|p{3cm}|p{3cm}|  }
 \hline
 \multicolumn{4}{|c|}{Nyldon Words} \\
 \hline
 0   &  10011   & 101111 &   1001111\\
 1 &   10110  & 1000000   &1011000\\
 10 &10111 & 1000001 &  1011001\\
 100    &100000 & 1000010 &  1011010\\
 101&   100001  & 1000011 &1011100\\
 1000& 100010  & 1000100   & 1011101\\
 1001& 100011  & 1000110& 1011110\\
 1011 & 100110 & 1000111 & 1011111 \\
 10000 & 100111 & 1001010 & \\
 10001 & 101100 & 1001100 & \\
 10010 & 101110 & 1001110 & \\
\hline
\end{tabular}
 \caption{List of Binary Nyldon Words of Length at Most 7}
 \label{tab:nyldontable}
 \end{table}
 
Our method will generalize to the following class of sets that includes Nyldon words.
\begin{definition}\label{def:nyldonlike}
Suppose we recursively generate a set of words $G$ and total order $\prec$ on $G$ as follows:
\begin{itemize}
    \item Each letter $a \in A$ is in $G$, with $\prec$ equal to the order $<$ on $A$.
    \item For $i=2, 3, \dots$, a word $w$ of length $i$ is in $G$ if $w$ cannot be factorized as $(w_1, w_2, \dots, w_k)$ where $w_1, \dots, w_k$ are in $G$ and $w_1 \preceq w_2 \preceq \cdots \preceq w_k$. Such a factorization is referred to as a $G$-factorization of $w$, with factors called $G$-factors. We refer to words $w$ in $G$ as $G$-words.
    \item Whenever we add a word to $G$, we keep the condition that for words $f, g \in G$, if $fg \in G$ then $f \prec fg$.
\end{itemize}
Then $(G, \prec)$ is a \textit{Nyldon-like set}. We refer to the latter condition as the \textit{Nyldon-like condition} for convenience.
\end{definition}
Note that the Nyldon words with the order $<_\text{lex}$ is a Nyldon-like set.
\begin{remark}\label{rmk:Gfactor}
If $w$ is in $G$, then its $G$-factorization is just $(w)$. Also each word must have at least one $G$-factorization. Given this fact, whether such a factorization is unique is equivalent to $G$ consisting of one word from each conjugacy class of primitive words by Schutzenberger's Theorem \cite{schutzen}; we prove that these facts hold for $G$ in Section 3. Nyldon-like sets will turn out to be right Hall sets (discussed at the end of Section 3) and therefore Lazard sets (discussed in Section 6), but the generation procedure used here is different than the Lazard procedure.
\end{remark}
\begin{example}
We generate an example Nyldon-like set $G$ on the binary alphabet.
\begin{enumerate}
    \item We start with $0 \prec 1$, where $A=\{0, 1\} \subset G$.
    \item The length-$2$ words $00$, $01$, $11$ can be $G$-factorized as $(0, 0), (0, 1), (1, 1)$, respectively. The word $10$ must be in $G$ with $10 \succ 1 \succ 0$ by the Nyldon-like condition. So, our current ordering of $G$ is $$0 \prec 1 \prec 10.$$
    \item The length-$3$ words $000, 001, 010, 011, 110, 111$ can be $G$-factorized as $(0, 0, 0), (0, 0, 1),$ $ (0, 10),$ $ (0, 1, 1), (1, 10), (1, 1, 1)$, respectively. The words $100$ and $101$ must be in $G$, and since $100=10 \cdot 0, 101=10 \cdot 1$, we have $100 \succ 10, 101 \succ 10$ by the Nyldon-like condition. We can choose how to order $100$ and $101$; let us pick $101 \prec 100$. So, our current ordering of $G$ is $$0 \prec 1 \prec 10 \prec 101 \prec 100.$$
    \item One can check that the length-$4$ words that cannot be $G$-factorized are $1000$, $1001$, $1011$. Let us insert these words into the ordering $\prec$ as $$0 \prec 1 \prec 10 \prec 101 \prec 1011 \prec 100 \prec 1000 \prec 1001,$$ respecting the Nyldon-like condition.
    \item One can check that the length-$5$ words that cannot be $G$-factorized are $10000, 10001, $ $10010, 10011, 10110, 10111$. We insert these words into the ordering $\prec$ in a way that respects the Nyldon-like condition, for example
    $$0 \prec 1 \prec 10 \prec 101 \prec 1011 \prec 10111 \prec 10110 \prec 100 $$ $$ \prec 1000\prec 10001 \prec 1001 \prec 10011 \prec 10000 \prec 10010.$$
    \item In contrast to Nyldon words, the length-$6$ word $101100$ can be $G$-factorized as $(101, 100)$, while $100101$ is in $G$. We can continue as above.
    \item We continue generating words in $G$ for all higher lengths, picking arbitrary choices to insert into the $\prec$ hierarchy that respect the Nyldon-like condition.
\end{enumerate}
\end{example}
We will prove \cref{lem:nyldonsuffix}, \cref{lem:longestsuffix}, and \cref{thm:nyldonprimitive} for all Nyldon-like sets in the following section.

\section{Melan\c{c}on's Algorithm}
A Nyldon-like set $G$ provides a unique $G$-factorization of any word $w$, among other properties. We will prove this fact by deriving an algorithm to find the $G$-factorization of a word. It will turn out that this algorithm coincides with an algorithm of Melan\c{c}on that works on Hall sets; we will elaborate at the end of this section.


\begin{definition}\label{def:preserve}
Suppose we can write $w$ as the concatenation of blocks $u_1, u_2, \dots, u_k$, and we can also factor it into $(n_1, n_2, \dots, n_m)$, i.e., $w=u_1u_2\cdots u_k=n_1n_2\cdots n_m$. We say that the factorization $(n_1, n_2, \dots, n_m)$ preserves the blocks $u_1, u_2, \dots, u_k$ if each factor $n_i$ can be written as the concatenation of a sequence of blocks in ($u_1, u_2, \dots, u_k$), i.e., no factor $n_i$ starts or ends in the middle of a block $u_j$. We can think of the factorization $(u_1, u_2, \dots, u_k)$ as a refinement of $(n_1, n_2, \dots, n_m)$.
\end{definition}

For the rest of this section, we fix a Nyldon-like set $(G, \prec)$.

\begin{lemma}\label{lem:firstblock}
Suppose $u_1, u_2, \dots, u_m$ are $G$-words such that for any substring of blocks $w=u_iu_{i+1}\cdots u_j$ with $1 \le i \le j \le m$,
\begin{enumerate}
    \item Any $G$-factorization of $w$ preserves the blocks $u_i, u_{i+1}, \dots, u_j$.
    \item If $w$ is in $G$ and $i>1$, then $w \succeq u_i$ (of course, equality holds only when $i=j$).
\end{enumerate}
Then no substring of blocks $u_au_{a+1}\cdots u_b$ with $1 \le a < b \le m$ can be in $G$ if $u_a \preceq u_{a+1}$.
\end{lemma}
\begin{proof}
Suppose $u_a \preceq u_{a+1}$. The word $u_{a+1}\cdots u_b$ has a $G$-factorization by \cref{rmk:Gfactor}, and by the assumption that the $G$-factorization preserves the blocks $u_{a+1}, \dots, u_b$, this factorization starts with some prefix $u_{a+1}\cdots u_{a+k}$. Then $u_{a+1}\cdots u_{a+k}$ is a $G$-word so $u_{a+1}\cdots u_{a+k} \succeq u_{a+1} \succeq u_a$, where the first inequality is due to the second condition of the lemma. Thus we can prepend the $G$-factor $u_a$ to this $G$-factorization of $u_{a+1}\cdots u_b$ to get a valid $G$-factorization of $u_a\cdots u_b$ with multiple $G$-factors, and so $u_a\cdots u_b$ is not in $G$. For equality to hold in the second condition, we need the words $w$ and $u_i$ to themselves be equal, or $i=j$.
\end{proof}

\begin{lemma}\label{lem:collapsing}
For $m \ge 2
$, suppose $u_1, u_2, \dots, u_m$ are $G$-words such that
\begin{enumerate}
    \item The word $w=u_1u_2\cdots u_m$ is in $G$.
    \item Any $G$-factorization of any substring of blocks $u_iu_{i+1}\cdots u_j$ with $1 \le i \le j \le m$ preserves the blocks $u_i, u_{i+1}, \dots, u_j$.
\end{enumerate}
Then $w \succ u_1$.
\end{lemma}
\begin{proof}
We will induct on $m$. For the base case $m=2$, by assumption $u_1, u_2, $ and $u_1u_2$ are in $G$, so by the Nyldon-like condition we have $w=u_1u_2 \succ u_1$, proving the lemma. Assume $m>2$ and that the lemma is true for all smaller $m$ at least $2$.

Let $k$ be the largest integer from $1$ to $m$ such that $u_1 \succ u_2 \succ \cdots \succ u_k$.

Let $u_2u_3\cdots u_m$ have a $G$-factorization that starts with $u_2u_3\cdots u_\ell$ for some $\ell$ sSuch an $\ell$ exists because all $G$-factorizations of $u_2u_3\cdots u_m$ preserve the blocks $u_2, u_3, \dots, u_m$). Then by the inductive hypothesis, we have $u_2u_3\cdots u_\ell \succeq u_2$. If we were to have $u_2 \succeq u_1$, then prepending $u_1$ to this $G$-factorization of $u_2u_3\cdots u_m$ would be a valid $G$-factorization of $w=u_1u_2\cdots u_m$, contradicting the fact that $w$ is in $G$. Then $u_1 \succ u_2$, so if $k$ be the largest integer from $1$ to $m$ such that $u_1 \succ u_2 \succ \cdots \succ u_k$, we have $k>1$. We claim that combining $u_{k-1}$ and $u_k$ into a block doesn't change the fact that any $G$-factorization of a substring of blocks in $(u_1, u_2, \dots, u_{k-2}, u_{k-1}u_k, u_{k+1}, \dots, u_m)$ preserves those blocks.

First, note that any $G$-factorization of $u_{k-1}u_k$ must preserve the blocks $u_{k-1}, u_k$ by assumption, so since $u_{k-1} \succ u_k$ the word $u_{k-1}u_k$ must be in $G$. Now if some $G$-factorization of a substring of blocks of $(u_1, u_2, \dots, u_{k-2}, u_{k-1}u_k, u_{k+1}, \dots, u_m)$ doesn't preserve the block $u_{k-1}u_k$, then we must have consecutive $G$-words $p, q$ in that $G$-factorization that are a witness to the lack of block preservation: we can write $p=u_{k-i}u_{k-i+1}\cdots u_{k-1}, q=u_ku_{k+1}\cdots u_{k+j}$ for some $i \ge 1, j \ge 0$ such that $p \preceq q$. We can apply \cref{lem:firstblock} to the $G$-words $u_k, u_{k+1}, \dots, u_{k+j}$ since the first condition is satisfied by assumption and the second condition is satisfied by the inductive hypothesis. This lemma tells us that if $j>0$, then $q$ cannot be in $G$ since $u_a \preceq u_{a+1}$; thus $q=u_k$. But then $p \succeq u_{k-i} \succ u_k=q$ (where the first inequality follows from the inductive hypothesis, and the second from $u_1 \succ u_2 \succ \cdots \succ u_k$), which contradicts the fact that $p$ and $q$ are consecutive $G$-factors.

So we can combine $u_{k-1}$ and $u_k$ into a block without changing the conditions of the lemma. But now we have $m-1$ blocks, so these blocks satisfy the inductive hypothesis and $w$ is greater than the first block under $\prec$. If $k=2$ then the first block is $u_1u_2$ and $w \succ u_1u_2 \succ u_1$, and if $k>2$ then the first block is $u_1$ and $w \succ u_1$, proving the lemma.
\end{proof}

\begin{definition}
Suppose $u_1, u_2, \dots, u_m$ are $G$-words such that for any sequence of numbers $a_1, a_2,$ $ \dots, a_i$ in $\{1, 2, \dots, m\}$, any $G$-factorization of $u_{a_1}u_{a_2}\cdots u_{a_i}$ preserves the blocks $u_{a_1}, \dots, u_{a_i}$. Then we say that the (multi)set $\{u_1, u_2, \dots, u_m\}$ satisfies the preservation condition. Note that whether a multiset satisfies the preservation condition only depends on whether its set of distinct elements does.
\end{definition}

Note that if a set satisfies the preservation condition, then any subset does as well.

\begin{corollary}\label{lem:smallestblock}
Suppose a multiset of $G$-words $\{u_1, u_2, \dots, u_m\}$ satisfies the preservation condition. Then, if $u_k$ is the smallest of the $G$-words $u_1, u_2, \dots, u_m$ under $\prec$, the word $u_ku_{a_1}u_{a_2}\cdots u_{a_i}$ is not in $G$ for any sequence $a_1, a_2, \dots, a_i$. Furthermore, any $G$-factorization of $u_ku_{a_1}u_{a_2}\cdots u_{a_i}$ begins with $u_k$.
\end{corollary}
\begin{proof}
The first condition of \cref{lem:firstblock} holds for the sequence of $G$-words $u_k, u_{a_1}, u_{a_2}, \dots, u_{a_i}$, and then by \cref{lem:collapsing} the second condition holds as well. Therefore since $u_k \preceq u_{a_1}$, by \cref{lem:firstblock} the word $u_ku_{a_1}u_{a_2}\cdots u_{a_i}$ cannot be in $G$, nor can any prefix of it that preserves the blocks $u_k, u_{a_1}, \dots, u_{a_i}$, except the prefix $u_k$ itself. Then any $G$-factorization of $u_ku_{a_1}u_{a_2}\cdots u_{a_i}$ must begin with $u_k$.
\end{proof}

\begin{example}
The set $A$ of single letters clearly satisfies the preservation condition, as does any multiset of single letters. \cref{lem:smallestblock} tells us that, for example, no word in $G$ starts with $0$ except $0$ itself. Note that in Table 1, every Nyldon word besides $0$ starts with $1$. We will be able to find more examples of sets satisfying the preservation condition by using the following lemma.
\end{example}

\begin{lemma}\label{lem:contraction}
Suppose a multiset of $G$-words $\{u_1, \dots, u_m\}$ satisfies the preservation condition with $u_k$ minimal among the $G$-words under $\prec$, with the additional condition that $u_{k-1} \succ u_k$ (rather than just $u_{k-1} \succeq u_k$). Then, if we combine $u_{k-1}$ and $u_k$ into a block, the resulting multiset of $m-1$ $G$-words still satisfies the preservation condition. 
\end{lemma}

\begin{proof}
First, by the hypothesis, any $G$-factorization of $u_{k-1}u_k$ must preserve the blocks $u_{k-1}, u_k$. Since $u_{k-1}u_k$ cannot have a $G$-factorization $(u_{k-1}, u_k)$, it must be in $G$.

We want to show that any $G$-factorization of $w=u_{a_1}u_{a_2}\cdots u_{a_i}u_{k-1}u_ku_{a_{i+1}}u_{a_{i+2}}\cdots u_{a_j}$ does not break up $u_{k-1}u_k$, for arbitrary $i, j \ge 0$ and $a_1, a_2, \dots, a_{i+j} \in \{1,\dots, m\}$. For sake of contradiction, suppose one does. Then, there are two $G$-factors, one ending with $u_{k-1}$, and one beginning with $u_k$ (since the blocks $u_{k-1}, u_k$ are preserved). Let these be $p$ and $q$, respectively. By \cref{lem:smallestblock}, $q=u_k$. Then either $p=u_{k-1}$, in which case $p \succ q$, or $p=u_{a_{i-b}}u_{a_{i-b+1}}\cdots u_{a_i}u_{k-1}$, in which case $p \succ u_{a_{i-b}} \succeq u_k$ by \cref{lem:collapsing}.

Either way, $p \succ q$, which is a contradiction since $p, q$ are consecutive $G$-factors.
\end{proof}

Let the operation of combining blocks as in \cref{lem:contraction} be called \textit{contraction}, where we always contract a block into the block to its left.
\begin{example}
Suppose $A=\{0, 1\}$ with $0 \prec 1$. The multiset $\{1, 1, 0, 0\}$ satisfies the preservation condition By \cref{lem:contraction}, the multiset $\{1, 10, 0\}$ also satisfies the preservation condition, and this multiset is obtained from the previous one from contracting the third block onto the second block. Any multiset consisting of copies of $0, 10, 1$ also satisfies the preservation condition by definition.
\end{example}

If we start with a word as a circular sequence of blocks which are initially letters, we can repeat this algorithm to end up with a $G$-word. It turns out that this algorithm is a special case of an algorithm of Melan\c{c}on; see \cite{melancon} for more information, and also the last section of \cite{reutenauer} or Chapter 4 of \cite{reutenauerhall} for an example of this algorithm. The algorithm is also presented at the end of \cite{nyldon}; we have rewritten the pseudocode from \cite{nyldon} below to work for a general Nyldon-like set.

Specifically, suppose we start with a primitive word, with each digit being a block. We consider the word as circular (with the blocks in a circle) rather than linear. At each step, we repeat the contraction operation, eventually terminating in one block, which is a $G$-word. As we will see shortly, this word is the unique $G$-word conjugate of the original word. The variable $T(i)$ designates the $i$th element of the list $T$ while $T(-i)$ denotes the $(n-i+1)$th element of $T$ if $n$ is the length of $T$.  \\

\begin{algorithmic}
\Require $w\in A^+$ primitive
\Ensure GC is the $G$-word conjugate of $w$
\State GC $\leftarrow$ list of letters of $w$,  $T\leftarrow$ list of letters of $w$
\While{length(GC) $>1$}
\If{$T(1)=\min_{\prec}$GC and $T(1)\prec$ $T(-1)$} 
\State $T\leftarrow (T(2),\ldots, T(-2), T(-1)\cdot T(1))$
 \EndIf
\State $i\leftarrow 2$, $j\leftarrow 2$
\While{$j\le \text{length}(T)$}
\While{$i\le \text{length}(T)$ and $T(i)\ne\min_{\prec}$GC}
\State $i\leftarrow i+1$
\EndWhile
\If{$i\le \text{length}(T)$ and $T(i)\prec$ $T(i-1)$}
\State $T\leftarrow (T(1),\ldots, T(i-1)\cdot T(i),\ldots,T(-1))$
\EndIf
\State $j\leftarrow i+1$, $i\leftarrow i+1$
\EndWhile
\State GC $\leftarrow T$
\EndWhile

\noindent\Return GC
\end{algorithmic}

\begin{center} Algorithm 1: Melan\c{c}on's algorithm \end{center}

We now provide an example of Melan\c{c}on's algorithm for the case of Nyldon words. 
\begin{example}\label{exa:nyldonconjugate}
Say we want to find the Nyldon conjugate of $10001011010101$. Recall that the ordering $\prec$ on Nyldon words is the lexicographic ordering.

\begin{enumerate}
    \item We start with $1, 0, 0, 0, 1, 0, 1, 1, 0, 1, 0, 1, 0, 1$ as a circular word (which we write here linearly for convenience).
    \item The smallest block lexicographically out of $0, 1$ is $0$. We contract the $0$s, giving $1000, 10, 1, 10,$ $ 10, 10, 1$.
    \item The smallest block lexicographically out of $1000, 10, 1$ is $1$. We contract the $1$s, giving $1000, 101, 10, 10, 101$.
    \item The smallest block lexicographically out of $1000, 101, 10$ is $10$. We contract the $10$s, giving $1000, 1011010, 101$.
    \item  The smallest block lexicographically out of $1000, 101, 1000$ is $1000$. We contract the $1000$. It is the first block in our linear representation of the word, but recall that we are treating this word (and therefore the sequence of blocks) as circular. So $1000$ is contracted onto $101$, giving the sequence of blocks $1011010, 1011000$.
    \item We contract the $1011000$ (the smaller block lexicographically), giving a Nyldon conjugate $10110101011000$.
\end{enumerate}

\end{example}

Note that in each step of \cref{exa:nyldonconjugate}, the resulting sequence of blocks forms a set satisfying the preservation condition. For example, $\{1000, 101, 10, 10, 101\}$ satisfies the preservation condition.

Using only the recursive definition of $G$-words and elementary methods, we prove the following:

\begin{theorem}\label{thm:primitive2}
\begin{enumerate}
    \item Every primitive word has exactly one $G$-word conjugate.
    \item No periodic word is in $G$.
\end{enumerate}
\end{theorem}

\begin{proof}
Suppose we start with a word as a circular sequence where each letter is its own block. Clearly any $G$-factorization of a sequence of these blocks preserves the blocks, as they are just single letters. We now perform Melan\c{c}on's algorithm on the sequence of blocks. The key observation is that the initial multiset of single letter blocks satisfies the preservation condition, and after any contraction step, the resulting multiset of blocks will still satisfy the preservation condition by \cref{lem:contraction}. The algorithm terminates when all the blocks are identical, as otherwise there is a smallest block under the total order $\prec$ on $G$ that has a different block to its left.

If the word is primitive, then when the algorithm terminates, we are left with one block $n$, which by \cref{lem:contraction} is in $G$ (since $\{n\}$ satisfies the preservation condition, meaning its $G$-factorization must be $(n)$). Suppose we start with a conjugate $w$ of $n$. We perform the algorithm on this $w$. Initially, this $w$ preserves the blocks. However, at some point the first block in $w$, which we denote $u$, will contract into the last block, as we must be left with the single block $n$ at the end. By \cref{lem:smallestblock}, $u$ cannot start $w$ if $w$ is a $G$-word. Therefore, $w$ is not in $G$ and every primitive word has exactly one $G$-word conjugate.

If we start with a periodic word with minimal unit $s$, suppose that the $G$-word conjugate of $s$ is $n$. We can think of our circular periodic word as a sequence of $k$ words $n$. Call each such word an $n$-group. If we perform Melan\c{c}on's algorithm on our word, then we know from above that until an $n$-group forms, its first block will not contract to the left. So all contractions between two blocks that happen will have the property that the two blocks are in the same $n$-group, until we end with $k$ blocks that are just $n$. Therefore, our periodic word is not in $G$.
\end{proof}

As we end up with the same result each time (in the case of a primitive word, its unique $G$-word conjugate), it doesn't matter in what order we do contractions of blocks of the same type, as long as we don't contract a block into an equal block to its left.

With a minor modification, we can use Melan\c{c}on's algorithm to find the $G$-factorization of a word too, which we illustrate in \cref{exa:factorization}.

\begin{definition}
We define the factorization version of Melan\c{c}on's algorithm on a word $w$ as follows: we start with each letter of $w$ as its own block, with the blocks in a line instead of a circle, and perform Melan\c{c}on's algorithm. Before each contraction step, we check if the current first block is the smallest under $\prec$. If so, we remove it, adding it to the $G$-factorization of $w$, and continue performing Melan\c{c}on's algorithm on the remaining blocks. This process continues until the whole word has been factored.
\end{definition}
We can use the factorization version of Melan\c{c}on's algorithm to reprove the fact from \cite{nyldon} that the $G$-factorization is unique.

\begin{theorem}\label{thm:melanconfactorization}
The factorization version of Melan\c{c}on's algorithm gives the unique $G$-factorization of a word $w$.
\end{theorem}
\begin{proof}
At a given step, suppose we have $k$ blocks $u_1, u_2, \dots, u_k$, so that the remaining (unfactorized) suffix of $w$ is $u_1u_2\cdots u_k$, and that these blocks satisfy the preservation condition. We have two cases based on which block is the smallest under $\prec$. \\

\noindent\textbf{Case 1}: Assume $u_1$ is smallest of the remaining blocks under $\prec$. Then by \cref{lem:smallestblock}, $u_1u_2\cdots u_i$ cannot be in $G$ for any $i \neq 1$. Since any $G$-factorization of $u_1u_2\cdots u_k$ preserves the blocks by the preservation condition, the factorization must start with $u_1$. We can take out $u_1$ as the only possible first  $G$-factor, and the resulting multiset of $k-1$ blocks $\{u_2, u_3, \dots, u_k\}$ still satisfies the preservation condition since it is a subset of the original multiset.\\

\noindent\textbf{Case 2}: Assume the smallest block under $\prec$ is not $u_1$. Then we can contract this smallest block, and the resulting multiset of blocks will still satisfy the preservation condition by \cref{lem:contraction}.\\

Since we begin with $w$ consisting of $|w|$ blocks, where each letter is its own block, the preservation condition is initially satisfied. Thus the preservation condition remains satisfied after each step, and we generate the only possible $G$-factorization of $w$. Since $w$ must have at least one $G$-factorization, the resulting factorization generated by this algorithm must be the unique valid $G$-factorization of $w$, and we are done.\end{proof}


We give an example of Melan\c{c}on's algorithm used for factorization, using the same string as the previous example. 
\begin{example}\label{exa:factorization}
Say we want to find the Nyldon factorization of the (linear) word $10001011010101$. As before, we contract the lexicographically smallest block(s) in each step.

\begin{enumerate}
    \item We start with $1, 0, 0, 0, 1, 0, 1, 1, 0, 1, 0, 1, 0, 1$.
    \item We contract the $0$s, giving $1000, 10, 1, 10, 10, 10, 1$.
    \item We contract the $1$s, giving $1000, 101, 10, 10, 101$.
    \item We contract the $10$s, giving $1000, 1011010, 101$.
    \item We contract the $1000$. It is the first block; here our treatment differs from the conjugate version of Melan\c{c}on's algorithm. The block $1000$ is taken out and becomes the first $G$-factor, leaving us with $1011010, 101$.
    \item We contract $101$, giving the block $1011010101$. We thus get a $G$-factorization of \\$(1000, 1011010101)$.
\end{enumerate}

\end{example}

We can also generalize \cref{lem:nyldonsuffix} and \cref{lem:longestsuffix} \cite[Theorems 13 and 14]{nyldon} to the case of any Nyldon-like set, using the following one.

\begin{lemma}\label{lem:melanconsuffix}
Let $s$ be a $G$-word suffix of $w$. Then, when applying the factorization version of Melan\c{c}on's algorithm to $w$, the last block will be $s$ at some step of the algorithm.
\end{lemma}

\begin{proof}
Suppose we perform the factorization version of Melan\c{c}on's algorithm on $w$. If we only look at the contractions that affect the last $|s|$ letters of $w$, then the sequence of contractions will be the same as performing Melan\c{c}on's algorithm on $s$. A block that is chosen to contract always contracts to its left, so letters that are to the left of $s$ will not affect it. The only scenario that would cause the sequence of contractions on $s$ to differ from Melan\c{c}on's algorithm on $s$ is if the leftmost block of $s$, which we call $u$, has to contract to the left, and $|u| < |s|$. But then $u$ would be the smallest block of $s$ under $\prec$, contradicting the fact that $s$ is in $G$ by \cref{lem:smallestblock} (which applies because the subset of blocks making up $s$ satisfies the preservation condition).

So, since $s$ is in $G$, it becomes one block when it is $G$-factorized, and thus becomes a fully formed block through the process of performing Melan\c{c}on's algorithm on $w$.
\end{proof}

\begin{corollary}\label{cor:suffix2}
Let $(u_1, u_2, \dots, u_k)$ be the $G$-factorization of a word $x$. Then $u_k$ is the longest $G$-word suffix of $x$. Also, for a $G$-word $x$ with a proper $G$-word suffix $s$, we have $s \prec x$.
\end{corollary}
\begin{proof}
If $s$ is a $G$-word suffix of a $G$-word $w$, then $s$ will form when applying Melan\c{c}on's algorithm to $w$ by \cref{lem:melanconsuffix}. A block $fg$ formed by contraction of $f, g$ satisfies $fg \succ f$ by the Nyldon-like condition and $f \succ g$ since $g$ is contracted making it minimal under $\prec$, so in particular $fg \succ g$. So if the sequence of suffix blocks of $w$ formed is $s_1, s_2, \dots, s_m=w$ in order, we have $s_1 \prec s_2 \prec \cdots \prec s_m$. Then $s \prec w$ as $s$ is one of the suffixes formed.

If we perform the factorization version of Melan\c{c}on's algorithm on a general word $w$, every $G$-word suffix of $w$ forms as a block at some point by \cref{lem:melanconsuffix}, so $w$'s longest $G$-word suffix will form. As the last factor in the $G$-factorization of $w$ must be a $G$-word suffix of $w$, and all such suffixes form as a block, the last factor is the longest $G$-word suffix of $w$.
\end{proof}

The algorithm we present coincides with a more general algorithm that works for all right Hall sets.

\begin{definition}
For an alphabet $A$, a factorization $(H, \prec)$ of the free monoid $A^*$ is a set of words $H$ with a total order $\prec$ such that any word $w$ has a unique factorization $w=h_1h_2\dots h_m$ for some $m \ge 1$ such that each $h_i \in H$ and $h_1 \preceq h_2 \preceq \dots \preceq h_m$.
\end{definition}

\begin{definition}
A factorization $(H, \prec)$ of the free monoid $A^*$ is a right Hall set if for all $f, g \in H$, whenever $fg \in H$ we have $fg \succ g$; a left Hall set if for all $f, g \in H$, whenever $fg \in H$ we have $fg \prec f$; and a Viennot set if it is both a left Hall set and a right Hall set.
\end{definition}

\begin{theorem}\cite{melancon}
The conjugate and factorization versions of Melan\c{c}on's algorithm both work on any right Hall set.
\end{theorem}

The Nyldon words form a right Hall set, as explained in Sections 11 and 12 of \cite{nyldon}. By proving unique factorization, we have shown that all Nyldon-like sets are right Hall sets. Traditionally, Hall sets are defined as sets of binary trees rather than words. Melan\c{c}on's algorithm was originally described as a ``rewriting algorithm'' on sequences of Hall trees in \cite{melancon}. It turns out that in any Hall set, each binary tree corresponds to a unique word. So, we can alternatively define Hall sets as sets of words rather than trees, as above. Right Hall sets are discussed in \cite{melancon} and \cite{reutenauerhall} (where they are just referred to as Hall sets). Viennot sets are discussed in \cite{viennot}, where left Hall sets are referred to as Hall sets. The fact that the above three definitions are valid definitions for Hall sets and Viennot sets, which were originally defined differently by using binary trees, is due to Viennot \cite[Proposition 1.8]{viennot}. Beware that certain authors may use the flipped version of $\prec$. Lyndon words form a Viennot set \cite[Example 54]{nyldon}.

Note that for a right Hall set $(H, \prec)$, whenever $f, g, fg \in H$, the conditions $fg \succ g$ and $f \succ g$ are guaranteed: the first by definition, and the second because otherwise $(f,g)$ would be a valid $H$-factorization of $fg$. Nyldon-like sets have the condition $fg \succ f \succ g$, while Viennot sets have the condition $f \succ fg \succ g$. So, Nyldon-like sets and Viennot sets can be thought of as opposite extremes with respect to the relation between $fg$ and $f$.

Our method gives a way to generate a set of factorizations of the free monoid (the Nyldon-like sets) recursively without proceeding through trees first (as Hall sets are traditionally defined). To our knowledge, such a recursive method has not previously appeared in the literature. For example, Melan\c{c}on proves that his algorithm works on Hall sets, but does not give a way to generate such sets \cite{melancon}; Reutenauer explains a way to generate Hall sets recursively by making trees, rather than through the lens of unique factorization of words \cite{reutenauerhall}. Furthermore, the proof that our method works is quite elementary, but it relies on the Nyldon-like condition. We wonder whether our recursive method works for general right Hall sets, and if so, whether a similarly elementary proof can be found.

\begin{question}
Suppose we replace the condition $f \prec fg$ with $g \prec fg$ in the third condition of \cref{def:nyldonlike}. Is the resulting set a factorization of the free monoid $A^*$ (and therefore a right Hall set), i.e. is factorization unique? If so, is there an elementary proof (similar to Section 3)?
\end{question}

\section{Applications of Melan\c{c}on's Algorithm}

In \cite{nyldon}, Charlier et al.\ ask (Open Problem 38) whether for a primitive word $w$ with Nyldon conjugate $n$, and large enough $k$ depending on $w$, the word $w^k$ can be factorized as $(p_1, p_2, \dots, p_i, $ $n, n, \dots, n, s_1, s_2, \dots, s_j)$ for some $K$, where there are $k-K$ copies of $n$ in the middle. We give a positive answer to this problem and provide a bound on $K$.

Furthermore, they ask whether the set of Nyldon words of a fixed length form a circular code (Open Problem 46). We explain the definition of a circular code below, and give a positive answer to the question.

We first demonstrate that there exists such a $K$ for powers of words $w$ in the specific case of Nyldon words.

\begin{lemma}\label{lem:prevalg}\cite{nyldon}
Suppose a Nyldon word $w$ has longest proper Nyldon suffix $x$. Then if we write $w=du_1u_2\cdots u_k$ where $d$ is a single letter and $(u_1, u_2, \dots, u_k)$ is the Nyldon factorization of $d^{-1}w$. Then, $du_1u_2\cdots u_i$ is Nyldon and $du_1u_2\cdots u_i >_{\text{lex}} u_{i+1}$ for all $1 \le i \le k$. Furthermore, $u_k = x$.
\end{lemma}
\begin{proof}
See the justification of Algorithm 1 in \cite{nyldon}, in the proof of Proposition 18.
\end{proof}

\begin{lemma}\label{lem:oneindexpairnyldon}
Suppose $uv$ is Nyldon for words $u, v$ with $u$ nonempty. Then, $u^kuv$ is not Nyldon for $k \ge 1$.
\end{lemma}
\begin{proof}
Without loss of generality, let $u$ be primitive.

For sake of contradiction, suppose $u^kuv$ is Nyldon. Then, suppose we start with $uv=x_0$ and add letters of $u^k$ to the left of $uv$ one at a time. For each $i$, let $y_i$ be the word the $i$th time the whole word is Nyldon, so $y_0=x_0$; in other words, $y_i$ is the $i$th shortest Nyldon suffix of $u^kuv$ that is longer than $uv=x_0$. Define $x_i$ to be $y_i(y_{i-1})^{-1}$, so $y_i=x_iy_{i-1}$.

We have $y_i=x_ix_{i-1}\cdots x_0$, and $y_{i-1}$ is the longest proper Nyldon suffix of $y_i$. Therefore, after $i$ steps, we have a word $x_ix_{i-1}\cdots x_0$ such that $x_jx_{j-1}\cdots x_0$ is Nyldon for all $j \le i$. By \cref{lem:prevalg}, $x_i$ is always Nyldon and $x_i >_\text{lex} y_{i-1} >_\text{lex} x_{i-1}$ for $i \ge 1$. At the end of this process, we have $u^kuv=x_mx_{m-1}\cdots x_0$ for some $m$, with $x_0=uv$. Then $x_m >_\text{lex} x_0$. We know that either $x_m$ is a prefix of $u$ or $u$ is a prefix of $x_m$, and since $x_m >_\text{lex} x_0=uv$, we must have that $|x_m|>|u|$ and $u$ is a prefix of $x_m$. For all $x_i$ with $0 < i < m$, we have that $x_i$ is in between $ab$ and $x_m$ lexicographically, in particular $x_1$. The only way this is possible is if $x_1$ starts with the same $|u|$ letters as $u$, and since $u$ is primitive, $x_1=u^j$ for some $1 \le j \le m$. But then if $j \ge 2$ we have that $x_1$ is not Nyldon, and if $x_1=u$ then $x_1 \le_\text{lex} x_0$. So, we have a contradiction and we are done.
\end{proof}

\begin{theorem}\label{thm:largeenoughpower}
For a primitive word $w$ with Nyldon conjugate $n$, and large enough $k$ depending on $w$, the word $w^k$ can be factorized as $(p_1, p_2, \dots, p_i, $ $n, n, \dots, n, s_1, s_2, \dots, s_j)$ for some $K$, where there are $k-K$ copies of $n$ in the middle.
\end{theorem}
\begin{proof}
Suppose $w=w_1w_2\cdots w_{\ell}$ are the letters in $w$, so $|w|=\ell$. By \cref{lem:oneindexpairnyldon}, in a word of the form $w^k$, there is a longest possible Nyldon factor that starts at $w_i$ and ends at $w_j$ for a fixed $i$ and $j$, since, if say $w_iw_{i+1}\cdots w_{\ell}w^aw_1w_2\cdots w_j$ is Nyldon, then $w_iw_{i+1}\cdots w_{\ell}w^bw_1w_2\cdots w_j$ cannot be Nyldon for $a\neq b$ and $a, b$ positive. So, there is a finite number of distinct possible words that can appear in the Nyldon factorization of a power of $w$. In fact, we can have at most $2$ factors starting at an index $i$ and ending at an index $j$ for given $i, j$. In addition to the possible $w_iw_{i+1}\cdots w_{\ell}w_1w_2\cdots w_j$, if $a$ is the smallest positive integer such that $w_iw_{i+1}\cdots w_{\ell}w^aw_1w_2\cdots w_j$ is Nyldon, then $w_iw_{i+1}\cdots w_{\ell}w^bw_1w_2\cdots w_j$ cannot be Nyldon for $b>a$ by \cref{lem:oneindexpairnyldon}, so only two Nyldon words can start and end at given indices $i, j$, respectively, modulo $\ell$. Furthermore, we cannot have a repeated such factor in the Nyldon factorization of $w$ unless $i\equiv j+1$ (mod $\ell$), as otherwise the starting indices of two consecutive factors are different. This scenario is only possible when the repeated factor is $n$. So, for large enough $k$, the word $n$ must appear as a Nyldon factor in $w^k$, as the total possible length of the $p_{\ell}$ and $s_m$ is bounded.
\end{proof}
The proof of \cref{lem:oneindexpairnyldon} (and therefore of \cref{thm:largeenoughpower}) uses the nature of lexicographic order and so does not easily generalize to a general Nyldon-like set. However using Melan\c{c}on's algorithm, we can provide a logarithmic bound on $K$ for any Nyldon-like set $(G, \prec)$ with some more work. We encourage the reader to follow \cref{exa:infection} to help with understanding the proof of \cref{thm:infectiongeneral}, especially the central claim of the proof.

\begin{theorem}\label{thm:infectiongeneral}
Fix a Nyldon-like set $(G, \prec)$. Suppose $n$ is a $G$-word. Let $s$ be a nonempty suffix of $n$ (possibly all of $n$). Then, for any word $a$, the word $sn^ka$ cannot be in $G$ if $k> \log_2(|n|).$ Furthermore, the $G$-factorization of $n^ka$ begins with a $G$-factor $x$ that is at least as long as $n$ such that $x \succeq n$.
\end{theorem}
\begin{proof}
We perform the factorization version of Melan\c{c}on's algorithm on $sn^ka$. The main idea is that if there are enough repeated blocks of a $G$-word in the middle of a word, then the influence of extra digits on the right cannot reach the leftmost block.

We have $k$ consecutive factors $n$ in our word; let us call these \textit{$n$-groups}. We aim to prove that the leftmost $n$-group fully forms as a block, that is, a block in the final factorization begins with the entire leftmost $n$-group. Say that an $n$-group has become \textit{infected} when a block to the right of the $n$-group contracts into the rightmost block of the $n$-group. In other words, an $n$-group becomes infected when it no longer preserves the blocks. An $n$-group that is not infected is called \textit{uninfected}. Note that uninfected $n$-groups are always identical and contract identically, so in particular at any point all uninfected $n$-groups will be composed of the same number of blocks.

We will prove the following central claim: consider two consecutive $n$-groups $n_1, n_2$. Note that before $n_1$ and $n_2$ are infected, they look identical. Suppose that they are made up of $m$ blocks each right before $n_2$ is infected (with respective blocks identical between the two $n$-groups). Then, $n_1$ cannot become infected while it consists of more than $\frac{m}{2}$ blocks.

Consider two consecutive $n$-groups in our word, which we hereafter refer to as $n_1$ and $n_2$ (with $n_1$ on the left). We look at the course of infection through $n_2$. We label a configuration of (uninfected $n_1$, infected $n_2$) with a triple (condition, blocks left, merge counter). We say that the leftmost infected block is the leftmost block in $n_2$ that is not identical to a corresponding block in $n_1$. All blocks to the right of this block are also considered to be infected. ``Blocks left'' is the number of blocks in $n_2$ to the left of the leftmost infected block (so essentially the number of blocks that are identical between the two $n$-groups on the left). ``Merge counter'' is the number of merges that have happened within $n_1$ after the infection reached $n_2$, and so it is also equal to the number of merges that have happened in any $n$-group to the left of $n_2$. Finally, if ``blocks left'' is $c$, we look at the $(c+1)$st blocks in each $n$-group. Call these $u_{c+1}$ in $n_1$ and $v_{c+1}$ in $n_2$. We set ``condition'' to ``greater'' if $v_{c+1} \succeq u_{c+1}$, and ``lesser'' if $u_{c+1} \succ v_{c+1}$. We start the infection of $n_2$ in a state (greater, $m-1$, $0$) if $n_1$ has $m$ blocks.

We now outline the possibilities for this state to change. Suppose we start in a state (condition, $c$, $d$). Let the blocks in $n_1$ be $u_1, u_2, \dots, u_a$ and the blocks in $n_2$ be $v_1, v_2, \dots, v_b$. Note that the first $c$ blocks in $n_1$, namely $u_1, u_2, \dots, u_c$, are respectively equal to the first $c$ blocks in $n_2$, namely $v_1, v_2, \dots, v_c$. If for some $i < c$ we have that $u_i$ and $u_{i+1}$ merge, then $v_i, v_{i+1}$ merge as well, and we are in a state (condition, $c-1$, $d+1$). This is Merge $1$ in the enumeration below.

If for some $i > c$ we have that $u_i, u_{i+1}$ merge, then ``merge counter'' goes up, while ``blocks left'' remains the same. It is possible for the merge to be with blocks $u_{c+1}$ and $u_{c+2}$. Consider this case. If $u_{c+1} \succ v_{c+1}$, then by the Nyldon-like condition, $u_{c+1}u_{c+2} \succ u_{c+1} \succ v_{c+1}$. So, a condition ``lesser'' remains the same. However, if $u_{c+1} \preceq v_{c+1}$, it is not immediately clear where $u_{c+1}u_{c+2}$ lies within this hierarchy. So, a condition ``greater'' could become ``lesser'', and for convenience we refer to it as ``unclear.'' These are Merges $2$ and $3$ in the enumeration below.

If for some $i >c$ we have that $v_i, v_{i+1}$, then nothing can change unless $i=c+1$. If $u_{c+1} \succ v_{c+1}$, then it is unclear where $v_{c+1}v_{c+2}$ goes in this hierarchy, so we say that ``lesser'' can become ``unclear''. Meanwhile if $v_{c+1} \succeq u_{c+1}$, then $v_{c+1}v_{c+2} \succ v_{c+1} \succeq u_{c+1}$, so ``greater'' remains the same. These are Merges $4$ and $5$ in the enumeration below.

The interesting situation is when merges happen involving $u_c$ and $u_{c+1}$ or $v_c$ and $v_{c+1}$; we say that these merges (and only these merges) spread the infection. Specifically, these merges spread the infection leftward by one block, as post-merge the $c$th blocks of each $n$-group are different. If $u_{c}$ and $u_{c+1}$ merge, and $v_{c}, v_{c+1}$ merge, then we end up in the state (unclear, $c-1$, $d+1$) (Merge 6). But perhaps a merge happens in one $n$-group and not the other. If $u_c$ merges with $u_{c+1}$ but $v_c$ does not merge with $v_{c+1}$, then since $u_cu_{c+1} \succ u_c=v_c,$ we end up in the state (lesser, $c-1$, $d+1$) (Merge 7). The last case (Merge 8) is if $v_c$ and $v_{c+1}$ merge, but $u_c$, $u_{c+1}$ do not. This phenomenon is only possible in the state ``lesser''. As $v_cv_{c+1} \succ v_c=u_c$, we go from (lesser, $c$, $d$) to (greater, $c-1$, $d$). Call this type of step \textit{quirky}.

Finally, the infection can only spread to $n_1$ if $v_1$ is smaller than $u_1$ under $\prec$. So, we must end in a ``lesser'' state. 

To summarize, the possible merges are
\begin{enumerate}
    \item ``merge counter'' goes up by $1$ and ``blocks left'' goes down by $1$
    \item ``merge counter'' goes up by $1$, condition goes from ``greater'' to ``unclear''
    \item ``merge counter'' goes up by $1$, condition remains at ``lesser''
    \item ``condition'' remains at ``greater''
    \item ``condition'' goes from ``lesser'' to ``unclear''
    \item ``merge counter'' goes up by $1$, ``blocks left'' goes down by $1$, and ``condition'' goes to ``unclear''
    \item ``merge counter'' goes up by $1$, ``blocks left'' goes down by $1$, and ``condition'' becomes ``lesser''
    \item ``blocks left'' goes down by $1$ and ``condition'' goes from ``lesser'' to ``greater'' (the quirky step)
\end{enumerate}

Pre-infection, when the two $n$-groups are identical, say we start with $m$ blocks in each group. Once $n_2$ gets infected, the $m$th block of $n_1$ is now shorter and smaller than the $m$th block of $n_2$, and we start in the state (greater, $m-1$, $0$).

Normally, ``merge counter'' goes up by $1$ whenever ``blocks left'' goes down by $1$ (Merges 1, 2, 3, 6, 7). The case when ``blocks left'' goes down by $1$ but ``merge counter'' does not increase is only possible in the quirky step (Merge 8), where condition changes from ``lesser'' to ``greater.'' So, if it takes $f$ steps to get to ``lesser'' for the first time, we go from (greater, $m-1$, $0$) to (lesser, $m-1-f$, $f$). Every quirky step is accompanied by changing the condition from ``lesser'' to ``greater,'' and to go back to ``lesser'' we must have some number of normal steps, each where merge counter goes up by $1$ and blocks left goes down by at most $1$, as Merges 4 and 5 cannot help in this regard. So, if we start at (shorter, $m-1-f$, $f$) and end at (shorter, $0$, $d$) for some $d$ the number of normal steps is at least the number of quirky steps. Therefore $d$ is at least $f+\frac{m-1-f}{2} \ge \frac{m}{2}$. Thus all $n$-groups to the left of the infected $n$-group $n_2$ undergo at least $\frac{m}{2}$ merges, so we end up with at most $\lfloor \frac{m}{2} \rfloor$ blocks in $n_1$ by the time the infection spreads, proving our central claim.

So, each time the infection spreads to a new $n$-group, the number of blocks in each uninfected $n$-group to its left has decreased by at least half. For example, if we have three consecutive $n$-groups $n_1, n_2, n_3$, and there are $m$ blocks in $n_2$ right before the infection spreads to it, there will be at most $\frac{m}{2}$ blocks in $n_1$ right before the infection spreads to it. Before any $n$-groups are infected, each $n$-group has at most $|n|$ blocks.

Thus before the infection spreads to any more than $\lfloor \log_2(|n|) \rfloor$ $n$-groups, an uninfected $n$-group will fully form, i.e., the $G$-word $n$ will appear as a block. This block will equal any $n$-groups to its left, and so by the Nyldon-like condition, it will remain at least as big if blocks contract onto it from the right. So, this fully formed, uninfected $n$-group cannot never merge to the left.

Therefore, the leftmost of the $k$ $n$-groups will fully form, and so appending the $G$-factorization of $s$ to the $G$-factorization of $n^ka$ will yield a valid $G$-factorization of $sn^ka$, as by \cref{cor:suffix2} the last $G$-factor in $s$ is smaller (under $\prec$) than the first $G$-factor of $n^ka$. Thus $sn^ka$ is not in $G$ and in particular, the leftmost $G$-factor $x$ of $n^ka$ will be $n$ with possible $G$-words merged onto it from the right. So $x \succeq n$ by the Nyldon-like condition and we are done.
\end{proof}

\begin{corollary}\label{cor:boundonpower}
Fix a Nyldon-like set $(G, \prec)$. For a primitive word $w$ with $G$-word conjugate $n$, and large enough $k$ depending on $w$, the word $w^k$ can be $G$-factorized as $$(p_1, p_2, \dots, p_i, n, n, \dots, n, s_1, s_2, \dots, s_j)$$ for $K \le \lfloor \log_2(|w|) \rfloor +2$, where there are $k-K$ copies of $n$ in the middle.
\end{corollary}
\begin{proof}
This corollary follows directly from \cref{thm:infectiongeneral}. The power $w^k$ is $sn^{k-1}a$ for some arbitrary $a$ and $s$ a suffix of $n$, so since the $G$-factorization of $n^{\lfloor \log_2(|w|) \rfloor+1}a$ begins with a $G$-factor $x \succeq n$, the $G$-factorization of $w^k=sn^{k-1}a$ begins with the $G$-factorization of $s$, followed by at least $k-1-(\lfloor \log_2(|w|) \rfloor+1)$ copies of $n$.
\end{proof}

Using the nature of lexicographic order, we can slightly improve the bound for Nyldon words.

\begin{theorem}\label{thm:boundonpower}
For a primitive word $w$ with Nyldon conjugate $n$, and large enough $k$ depending on $w$, the word $w^k$ can be Nyldon factorized as $$(p_1, p_2, \dots, p_i, n, n, \dots, n, s_1, s_2, \dots, s_j)$$ for $K \le \lfloor \log_2(|w|) \rfloor +1$, where there are $k-K$ copies of $n$ in the middle.
\end{theorem}
\begin{proof}
We exactly follow the proof of \cref{thm:infectiongeneral}, but are able to get an improvement of $1$.

Suppose we use Melan\c{c}on's algorithm on a word $w^k$ for $k \ge \lfloor \log_2(|w|) \rfloor +1$. We then have $\lfloor \log_2(|w|) \rfloor$ $n$-groups in the middle of our word, with some suffix on the right. Suppose that when something from the right suffix contracts onto the rightmost $n$-group, that $n$-group consists of $x$ blocks. Then, each time a $n$-group is infected, the number of blocks in the current $n$-group must decrease by at least half. If the number of $n$-groups is at least $\lfloor \log_2(x) \rfloor +1$, then the leftmost $n$-group will fully form, and so not be able to contract onto any block to its left, as such a block is a Nyldon suffix of $n$, and therefore lexicographically less than $n$ by \cref{lem:nyldonsuffix}.

Clearly $n$ starts with its maximal letter, so the right suffix (after the $k-1$ blocks of $n$ in $w^k$) does as well. So, by the time the maximal digit has been contracted in Melan\c{c}on's algorithm, all digits will have been contracted and there will be no blocks left of length $1$. Thus, the total number of blocks in the rightmost $n$-group, $x$, is at most $\frac{|w|}{2}$ when something contracts onto it. The infection then spreads to at most $\lfloor \log_2(x) \rfloor+1 \le \lfloor \log_2(w) \rfloor$ blocks (instead of at most $\lfloor \log_2(w) \rfloor+1$ blocks), giving a bound $K \le \lfloor \log_2(w) \rfloor+1$.
\end{proof}

\begin{example}\label{exa:infection}
We provide an example of how the above proofs work, by showing Melan\c{c}on's algorithm on a specific power in the context of Nyldon words. Consider the word $$w=01111011011111011110111,$$ with Nyldon conjugate $$n=1w1^{-1}=10111101101111101111011.$$ We show that in this case, $K=4$ (as shown in \cite{nyldon}), by factorizing $w^5$.

Note that $w^5=0111101101111101111011n^41$. The prefix $0111101101111101111011$ will be factorized and stay the same in front of $n$, as its factorization ends with a proper Nyldon suffix of $w$ which is smaller than $w$ under $<_{\text{lex}}$ by \cref{lem:nyldonsuffix}, so we focus only on factorizing $n^41$. We use semicolons to show the barriers between the $n$ factors. In each step, we contract the lexicographically smallest block(s).

\begin{enumerate}
    \item Start with $n^41=\\1,0,1,1,1,1,0,1,1,0,1,1,1,1,1,0,1,1,1,1,0,1,1;\\1,0,1,1,1,1,0,1,1,0,1,1,1,1,1,0,1,1,1,1,0,1,1;\\1,0,1,1,1,1,0,1,1,0,1,1,1,1,1,0,1,1,1,1,0,1,1;\\1,0,1,1,1,1,0,1,1,0,1,1,1,1,1,0,1,1,1,1,0,1,1;\\1$.
    \item Contract the $0$s to get \\
    $10, 1, 1, 1, 10, 1, 10, 1, 1, 1, 1, 10, 1, 1, 1, 10, 1, 1;\\
    10, 1, 1, 1, 10, 1, 10, 1, 1, 1, 1, 10, 1, 1, 1, 10, 1, 1;\\
    10, 1, 1, 1, 10, 1, 10, 1, 1, 1, 1, 10, 1, 1, 1, 10, 1, 1;\\
    10, 1, 1, 1, 10, 1, 10, 1, 1, 1, 1, 10, 1, 1, 1, 10, 1, 1;\\
    1$.
    \item Contract the $1$s to get \\
    $10111, 101, 101111, 10111, 1011; 10111, 101, 101111, 10111, 1011;\\ 10111, 101, 101111, 10111, 1011; 10111, 101, 101111, 10111, 10111$.\\
    The last of the four $n$-groups is different from the first three, so the infection starts. The number of blocks per (uninfected) $n$-group is $5$ (which is true right before the infection starts). Our state (condition, blocks left, merge counter) is (greater, $4$, $0$), where ``blocks left'' is $1$ less then the initial number of blocks in the left $n$-group. We look at the fourth $n$-group, where the infection is spreading, and the uninfected third $n$-group to its left.
    \item Contract the $101$s, to get \\
    $10111101, 101111, 10111, 1011; 10111101, 101111, 10111, 1011;\\ 10111101, 101111, 10111, 1011; 10111101, 101111, 10111, 10111$.\\ The number of blocks per (uninfected) $n$-group is $4$. This merge corresponds to Merge 1 in the enumeration, and our state (condition, blocks left, merge counter) changes to (greater, $3$, $1$).
    \item Contract the $1011$s to get \\
    $10111101, 101111, 101111011; 10111101, 101111, 101111011;\\ 10111101, 101111, 101111011; 10111101, 101111, 10111, 10111$.\\ The number of blocks per uninfected $n$-group is $3$. This merge spreads the infection and corresponds to Merge 7. Our state (condition, blocks left, merge counter) changes to (lesser, $2$, $2$).
    \item Contract the first $10111$ to get \\
     $10111101, 101111, 101111011; 10111101, 101111, 101111011;\\ 10111101, 101111, 101111011; 10111101, 10111110111, 10111$. \\This merge corresponds to Merge 8 and changes our state (condition, blocks left, merge counter) is (greater, $1$, $2$).
     \item Contract the second $10111$ to get \\
     $10111101, 101111, 101111011; 10111101, 101111, 101111011;\\ 10111101, 101111, 101111011; 10111101, 1011111011110111$.\\ This merge corresponds to Merge 4 and has no effect on the state.
     \item Contract the $101111$s to get \\
     $10111101101111, 101111011; 10111101101111, 101111011;\\ 10111101101111, 101111011; 10111101, 1011111011110111$.\\ The number of blocks per uninfected $n$-group is $2$. This merge corresponds to Merge 7 and our state (condition, blocks left, merge counter) changes to (lesser, $0$, $3$).
     \item Contract the $10111101$ to get \\
     $10111101101111, 101111011; 10111101101111, 101111011;\\ 10111101101111, 10111101110111101, 1011111011110111$.\\ The infection has spread to the next $n$-group. Note that the number of blocks per uninfected $n$-group decreased from $5$ to $2$ over the course of infection of the previous $n$-group (since ``merge counter'' was $3$ at the end), and $2 \le 5/2 $, demonstrating the bound proved in the central claim.
     \item Contract the $101111011$s to get\\ 
     $10111101101111101111011; 10111101101111101111011; \\ 10111101101111, 10111101110111101, 1011111011110111$.\\ We now have $n$ fully formed! The next step will create \\$10111101101111101111011;\\ 1011110110111110111101110111101101111, 10111101110111101, 1011111011110111$,\\ but the leftmost $n$ factor will remain unaltered.
\end{enumerate}

As only one $n$ appears in the Nyldon factorization of $w^5$ in this example, we have $K=4$.
\end{example}

\begin{definition}\label{def:code}
A subset $F$ of $A^*$ is a \textit{code} if all possible concatenations of (not necessarily distinct) words in $F$ yield distinct words. The subset $F$ is a \textit{circular code} if for any words $u$ and $v$, we have $uv, vu \in F^*$ implies $u, v \in F^*$, where $F^*$ is the set of all possible concatenations of words in $F$.
\end{definition}

In other words, a code is circular if whenever we take a concatenation of words in the code and put the string in a circle, we can recover the original sequence of words. For example, $\{00, 01, 10\}$ is not a circular code, as $(00, 10)$ forms the same circular word as $(01, 00)$. It turns out that Lyndon words over a given alphabet of a fixed length form a circular code (for example, see Exercise 8.1.5 of \cite{codebook}). We prove a similar result for Nyldon words.

For the rest of this section, we fix a right Hall set $(H, \prec)$. We refer to a word in $H$ as an $H$-word.

\begin{lemma}\label{lem:crossingboundary}
Consider a sequence of $H$-words $w_1, w_2, \dots, w_k$. Suppose we perform Melan\c{c}on's algorithm on the circular word $w_1w_2\cdots w_k$ (where blocks are initially individual letters). Then at least one block $w_j$ will form before any block crosses a boundary between two words $w_i, w_{i+1}$ (where $w_{k+1}=w_1$).
\end{lemma}
\begin{proof}
For sake of contradiction, suppose that a block crosses the boundary between $w_{i-1}, w_i$ without any word $w_j$ fully forming. Right before this happens, because by assumption no word $w_j$ is fully formed, let the first block in each $w_j$ be $p_j$ and the last block be $s_j$.

We have that $p_{i}$ is being joined to $s_{i-1}$, so $p_{i}$ is the smallest of all the blocks still remaining under $\prec$. In particular, $p_{i}$ is the smallest of the blocks still remaining under $\prec$ that make up $w_{i}$, of which there are at least two. But then, if we were to apply the conjugate version of Melan\c{c}on's algorithm to (the circular word) $w_i$, at some point the block $p_i$ would contract to the left. Then the final $H$-word conjugate remaining would not be $w_i$, which is a contradiction. Therefore, blocks will never cross the original $H$-word boundaries until strictly after at least one is formed.
\end{proof}

\begin{lemma}\label{lem:lengthlblock}
Consider a sequence of $H$-words $w_1, w_2, \dots, w_k$, each of length $\ell$. Suppose we perform Melan\c{c}on's algorithm on the circular word $w_1w_2\cdots w_k$ (where blocks are initially individual letters). Then the first block to form of length at least $\ell$ is some $w_j$.
\end{lemma}
\begin{proof}
Any block of length at least $\ell$ that is not one of the $w_j$'s must cross a boundary between some $w_i, w_{i+1}$. By \cref{lem:crossingboundary}, no such block can form until some $w_j$ forms. 
\end{proof}

\begin{theorem}\label{thm:circularcodegeneralhall}
The $H$-words of any fixed length $\ell$ form a circular code.
\end{theorem}
\begin{proof}
For sake of contradiction, suppose we have a sequence of $H$-words $w_1, w_2, \dots, w_k$, each of length $\ell$, such that circularly shifting $w_1w_2\cdots w_k$ by a non-multiple of $\ell$ gives $v_1v_2\cdots v_k$, where the $v_i$ are all $H$-words of length $\ell$. Suppose we perform Melan\c{c}on's algorithm on the circular word $w_1w_2\cdots w_k$, which is the same as the circular word $v_1v_2\cdots v_k$. By \cref{lem:lengthlblock}, the first block of length at least $\ell$ formed is still some $w_i$. This means that a boundary is crossed in the $v_i$'s before any block $v_i$ is formed, which is impossible by \cref{lem:crossingboundary}. Therefore, there is no possible other sequence $v_1v_2\cdots v_k$ that forms the same circular word as $w_1w_2\cdots w_k$.
\end{proof}

\section{Another Algorithm}
In \cite{nyldon}, Charlier et al.\ provide an algorithm for computing the Nyldon factorization of a word. We prove that this algorithm is linear in the length of the word. Thus the Nyldon factorization can be computed in linear time, just like for the Lyndon factorization, as shown by Duval \cite{duval}. We reproduce the algorithm here. \\

\begin{algorithmic}
\Require $w\in A^+$
\Ensure NylF is the Nyldon factorization of $w$
\State  $n\leftarrow \text{length}(w)$, NylF $\leftarrow (w[n])$
\For{$i=1$ to $n-1$}
\State $\text{NylF} \leftarrow (w[n-i],\text{NylF})$
\While{$\text{length}(\text{NylF}) \ge 2  \text{ and } \text{NylF}(1)>_{\text{lex}}\text{NylF}(2)$}
\State $\text{NylF}\leftarrow (\text{NylF}(1)\cdot \text{NylF}(2),\,\text{NylF}(3),\, \ldots,\,\text{NylF}(-1))$
\EndWhile
\EndFor

\noindent\Return NylF
\end{algorithmic}

\begin{center} Algorithm 2 \cite{nyldon}: Computing the Nyldon factorization \end{center}

The number of initial lexicographic comparisons is equal to the number of digits in $w$, and the number of additional lexicographic comparisons is at most the number of times two words combine into a bigger word. There are $|w|-1$ barriers between words, so the total number of lexicographical comparisons is at most $2|w|-1$, which is linear in $w$. One way to do comparisons fast is with a suffix array, which takes $O(|w|)$ time to construct, giving a least common prefix (LCP) array \cite{suffix, lcp}. Comparing substrings of $w$ is equivalent to a range minimum query (RMQ) on the LCP array between the indices representing the suffixes where the substrings start. We can do RMQ in constant time with linear preprocessing time \cite{fischerrmq}. Therefore, we obtain a runtime of $O(|w|)$.

\begin{remark}
This algorithm also finds the factorization of a word with respect to any right Hall set $(H, \prec)$, if we replace the comparison $>_\text{lex}$ with $\succ$ in the \textbf{while} loop. The proof is the same as that of Proposition 18 in \cite{nyldon}, as Theorem 14 and Proposition 17 in \cite{nyldon} hold for all right Hall sets. For example, the general form of Proposition 17 is equivalent to Lemma 1.8 of \cite{viennot}, and implies the general form of Theorem 14.
\end{remark}

We can also use Melan\c{c}on's algorithm to factorize words. If we store the current blocks in a heap, and simply contract the current minimum each time, we get a time complexity of the previous algorithm increased by a logarithmic factor. So, Melan\c{c}on's algorithm for factorization or finding the Nyldon conjugate has a runtime of $O(|w|\log|w|)$.

\section{The Lazard Procedure}
Let $A^{\le n}$ be the set of words on $A$ with length at most $n$. Also, for a set of words $X$ and word $w$, let $Xw^*=X \cup Xw \cup Xww \cup \cdots$, where $Xw^i$ denotes the set of words $\{xw^i : x \in X\}$ for all $i \ge 0$.

\begin{definition}\label{def:lazard}
A right Lazard set is a subset $F$ of $A^+$ with a total order $<$ satisfying the following property for each positive integer $n$: suppose $F \cap A^{\le n}=\{ u_1, u_2, \dots, u_k\}$ with $u_1 < u_2 < \cdots  < u_k$. Let $Y_i$ be a sequence of sets defined as $Y_1=A$ and $Y_i=(Y_{i-1} \backslash u_{i-1})u_{i-1}^*$ for $i \ge 2$. Then, for all $1 \le i \le k$, we have that $u_i \in Y_i$, and $Y_k \cap A^{\le n}=\{u_k\}$.
\end{definition}

The \textit{Lazard procedure of length} $n$ is the act of generating the $Y_i$ by choosing $u_i$ in $Y_i$, removing it from $Y_i$, and creating $Y_{i+1}$, all for a given value of $n$. We can think of the total order on $F \cap A^{\le n}$ (and by extension, on $F$) as being induced by the procedure itself by the choice of $u_i$.

Lazard sets are discussed in \cite{viennot}, and are proved to be equivalent to Hall sets. In \cite{nyldon}, Charlier et al.\ prove that the Nyldon words form a right Lazard set, and conjecture that the right Lazard procedure generates all the Nyldon words up to a certain length much before the procedure ends, unlike the procedure for Lyndon words (Open Problem 59). We explicitly determine the number of steps it takes to generate all the Nyldon words up to a given length. For reference, we reproduce the example in \cite{nyldon} of the right Lazard procedure on binary Nyldon words of length at most $5$. Note that all words have been generated by the fourth step.

\begin{table}
\centering
	\begin{tabular}{c || l | l}
	$i$ 	& $Y_i \cap \{{ 0,1}\}^{\le 5}$											& $u_i$ \\
	\hline 
	1 	& \{{0, 1}\} 																& {0} \\
	2 	& \{{1, 10, 100, 1000, 10000}\} 											& {1} \\
	3 	& \{{10, 101, 1011, 10111, 100, 1001, 10011, 1000, 10001, 10000}\}  				& {10} \\
	4 	& \{{101, 10110, 1011, 10111, 100, 10010, 1001, 10011, 1000, 10001, 10000}\}  	& {100}\\
	5 	& \{{101, 10110, 1011, 10111, 10010, 1001, 10011, 1000, 10001, 10000}\} 		& {1000}\\
	6 	& \{{101, 10110, 1011, 10111, 10010, 1001, 10011, 10001, 10000}\} 				& {10000}\\
	7 	& \{{101, 10110, 1011, 10111, 10010, 1001, 10011, 10001}\}  					& {10001}\\
	8 	& \{{101, 10110, 1011, 10111, 10010, 1001, 10011}\} 							& {1001}\\
	9 	& \{{101, 10110, 1011, 10111, 10010, 10011}\} 								& {10010} \\
	10 	& \{{101, 10110, 1011, 10111, 10011}\}  										& {10011}\\
	11 	& \{{101, 10110, 1011, 10111}\}  											& {101} \\
	12 	& \{{10110, 1011, 10111}\} 						 						& {1011} \\
	13 	& \{{10110, 10111}\} 				 									& {10110}\\
	14 	& \{{10111}\} 															& {10111}\\
	\end{tabular}
\caption{An example of the Lazard procedure with $n=5$ \cite{nyldon}.}
\end{table}

\begin{lemma}\label{lem:lazardunique}
Every Nyldon word $w$ can be generated in exactly one way by the Lazard procedure of length $n$ for $n \ge |w|$, which is the same for all such $n$.
\end{lemma}
\begin{proof}
Each word in some $Y_i$ is generated by adding Nyldon words to the right of an initial letter in $Y_1$, with words added in nondecreasing lexicographic order, by \cref{def:lazard}. Then since the Nyldon words form a right Lazard set, each Nyldon word $w$ arises as a word in some $Y_i$ (assuming we perform the procedure for $n \ge |w|$). Thus we can write every Nyldon word $w$ that is not a single letter as $au_{i_1}u_{i_2}\cdots u_{i_k}$ where $a \in A$ and $i_1 \le i_2 \le \cdots  \le i_k$ as it is generated by the procedure, so $u_{i_1} \preceq u_{i_2} \preceq \cdots \preceq u_{i_k}$ are Nyldon words. However, $u_{i_1}u_{i_2}\cdots u_{i_k}$ is then the unique Nyldon factorization of $a^{-1}w$, so there is only one way to generate $w$ during the Lazard procedure (which is the same for all $n$). The word $w$ is indeed kept when intersecting with $A^{\le n}$ for $n \ge |w|$, and all the words $u_{i_j}$ have length at most $n$, so $w$ is indeed generated in the Lazard procedure of length $n$ for $n \ge |w|$.
\end{proof}

Now, we determine the step at which all Nyldon words have been generated. For the rest of this section, we fix an alphabet $A$ of size $m+1$, so $A=\{0,1,\dots ,m\}$.

\begin{lemma}\label{lem:largestnyldonword}
The largest Nyldon word lexicographically of length at most $\ell$ is $m(m-1)m^{\ell-2}$.
\end{lemma}
\begin{proof}
No Nyldon word $w$ can start with $mm$, as adding $m$ to the Nyldon factorization of $m^{-1}w$ yields a valid Nyldon factorization with at least $2$ factors. So, $m(m-1)m^{\ell-2}$ is lexicographically the largest possible such word remaining. This word is clearly Nyldon by any factorization algorithm, so we are done.
\end{proof}

\begin{proposition}\label{prop:firstblock}
For a positive integer $\ell>1$ and (possibly empty) word $v$, when we perform Melan\c{c}on's algorithm on $w=m(m-1)m^{\ell}v$, a block of length $\ell+1$ will form at the beginning of $w$ at some step of the algorithm.
\end{proposition}
\begin{proof}
Write $w=m(m-1)m^{\ell-1}(mv)$. We claim that the Nyldon prefix $u=m(m-1)m^{\ell-1}$ will form a block at some step. For sake of contradiction, suppose that $u$ does not form a block at any step. Since $u$ is Nyldon, it would naturally form a block on its own under Melan\c{c}on's algorithm: first the letter $m-1$ is contracted, then each of the $\ell-1$ copies of $m$ is contracted. The only way that $u$ does not form a block is if it is interrupted by $mv$: specifically, the leftmost block of $mv$ contracts to the left onto the rightmost block of $u$. But until $u$ forms a block, its rightmost block is always the letter $m$, which is at most as big lexicographically as the leftmost block of $mv$. Thus we have a contradiction and we are done.
\end{proof}

\begin{lemma}\label{lem:repunitnyldonword}
Let $\ell$ be a nonnegative integer. If a word $v$ has length at most $\ell$, then $w=m(m-1)m^{\ell+1}v$ is Nyldon.
\end{lemma}
\begin{proof}
If we perform Melan\c{c}on's algorithm on $w$, then a block of length $\ell+2$ will form at the beginning of the word by \cref{prop:firstblock}. This block is lexicographically greater than any block to the right, since any such block has length less then $\ell+2$ (as the whole word has length at most $2\ell+3$), and our first block is the lexicographically largest Nyldon word of length at most $\ell+2$ by \cref{lem:largestnyldonword}. So, the whole word is Nyldon.
\end{proof}

\begin{proposition}\label{prop:lazardstopword}
Let $\ell \ge 5$ be a positive integer, and suppose that the step at which the Lazard procedure of length $\ell$ generates all the Nyldon words of length up to $\ell$ is the step corresponding to (the removal of) the Nyldon word $u_i$ for some $i$. If $\ell$ is odd, then $u_i=m(m-1)m^{\frac{\ell-5}{2}}$, and if $\ell$ is even, then $u_i=m(m-1)m^{\frac{\ell-6}{2}}(m-1)$.
\end{proposition}

\begin{proof}
The main idea is that the chosen $u_i$ is the lexicographically largest Nyldon word that can still be affixed to the end of a larger Nyldon word within the length limits. Suppose that by the time $u_i$'s step happens, there is still some Nyldon word that has not yet appeared. Then, it must equal $ab$ where $a, b$ are Nyldon, and $a >_{\text{lex}} b >_{\text{lex}} u_i$.

If $\ell$ is odd, then since $a, b$ are greater than $u_i$ lexicographically, they each must have a length at least $\frac{\ell+1}{2}$ since $u_i$ is the largest Nyldon word lexicographically with a length of at most $\frac{\ell-1}{2}$ by \cref{lem:largestnyldonword}. Then $|ab| \ge \ell+1$, which is a contradiction.

If $\ell$ is even, then $a$ and $b$, being greater than $u_i$ lexicographically, must each have length at least $\frac{\ell}{2}$. Since $|ab|$ must be at most $\ell$, the only possibility is then $a=b=u_i(m-1)^{-1}m$, which is impossible because $a$ and $b$ cannot be equal. So, we have proved that by step $u_i$, every Nyldon word is generated.

Furthermore, at step $u_i$, if $\ell$ is odd, then $(u_im)u_i$ is generated, and if $\ell$ is even, then $(u_i(m-1)^{-1}m)u_i$ is generated, so a new word is generated. So, we are done.
\end{proof}

Now, we want to calculate how far from the end the Lazard procedure has generated all the Nyldon words, or equivalently, calculate how many Nyldon words of length at most $\ell$ are lexicographically greater than the $u_i$ in \cref{prop:lazardstopword}. We examine two cases, based on whether $\ell$ is even or odd.

\begin{proposition}\label{prop:lazardstopnumberodd}
Suppose $\ell=2n+1$ for $n \ge 7$, and let $u=m(m-1)m\cdots m$, where $u$ has length $n$, so there are $n-2$ copies of $m$ at the end. Then, the number of Nyldon words lexicographically greater than $u$ and with length at most $\ell$ is 
$\displaystyle \frac{(m+1)^{n+2}-(m+1)}{(m+1)-1}\ - ((m+1)^3+(m+1)^2+2(m+1)+2)$.
\end{proposition}

\begin{proof}
Suppose $\ell=2n+1$ for $n \ge 7$. Most words of length at most $\ell$ beginning with $u$ are Nyldon, so we will instead count the words which are not. Let $w$ be a word beginning with $u$ and of length at most $\ell$ such that $w$ is not Nyldon. We split into cases based on the $(n+1)$th letter of $w$, where each time we perform Melan\c{c}on's algorithm. Note that $|w|>|u|$ since $w$ is not Nyldon.\\

\noindent\textbf{Case 1}: Suppose $w$ begins with $um$. Then, a block $u$ will form at the beginning of the word by \cref{prop:firstblock}. For $w$ to not be Nyldon, the Nyldon block to the right of $u$ must be lexicographically at least $u$ at some step of the algorithm. Otherwise, it would contract into the first block, and any remaining block would have a length at most $n$ and therefore be lexicographically smaller than the leftmost block by \cref{lem:largestnyldonword}, and so would eventually contract to the left.

By \cref{lem:largestnyldonword}, this block must have length $n$ or $n+1$, and in particular must either be $u$ or $ud$ for a letter $d$. There are $|A|+1=(m+1)+1$ possibilities in this case, depending on whether $w$ has length $2n+1$ or $2n$, respectively.\\

\noindent\textbf{Case 2}: Suppose $w$ begins with $u(m-1)$. Then, a block of $um^{-1}$ will form at the beginning of the word by \cref{prop:firstblock}, which is the largest Nyldon word lexicographically of length at most $n-1$ by \cref{lem:largestnyldonword}. To the right of that, a block starting with $m(m-1)$, and therefore of length at least $2$, will form. We claim that this ``middle'' block cannot combine with the block to its left. Suppose it does. Then, the leftmost block will be of length at least $n+1$ and start with $u$, and so no Nyldon block to the right of it can be lexicographically greater than it, so the whole word $w$ would be Nyldon.

So at some step of the algorithm, the ``middle'' block, i.e. the one starting at the $n$th index, must be lexicographically at least $um^{-1}$. Therefore, the word made up from indices $n$ through $(2n-2)$ must be another copy of $um^{-1}$ (using the lexicographic maximality of $um^{-1}$). Then $w=(um^{-1})(um^{-1})v$ for some word $v$ of length $3$. By \cref{lem:repunitnyldonword}, the word $(um^{-1})v$ is Nyldon if $um^{-1}$ ends with at least $4$ copies of the letter $m$, or $n \ge 7$; in this case $(um^{-1}, um^{-1}v)$ is the Nyldon factorization of $w$. Since any $v$ of length at most $3$ works, we get $1+(m+1)+(m+1)^2+(m+1)^3$ possible words that are not Nyldon, summing over $v$ of length $0, 1, 2, $ and $3$ (as there are $(m+1)^i$ possibilities for $v$ of length $i$).\\

\noindent\textbf{Case 3}: Suppose $w$ begins with $ud$ for a letter $d \le m-2$. When performing Melan\c{c}on's algorithm, a block of $um^{-1}$ will form at the beginning of the word by \cref{prop:firstblock}, and will be lexicographically greater than the block to its right. When they combine, the block will start with $u$, and be greater than any block to the right. So, the whole word would be Nyldon. Thus there are $0$ possible words in this case. \\

Therefore, the total number of words lexicographically greater than $u$ with length at most $\ell=2n+1$ which are not Nyldon, for $n \ge 7$, is $((m+1)+1)+((m+1)^3+(m+1)^2+(m+1)+1)=(m+1)^3+(m+1)^2+2(m+1)+2$, obtained by summing over all cases. There are $(m+1)^k$ words of length $n+k$ starting with $u$ for $1 \le k \le n+1$, and all are Nyldon except the aforementioned words. So, using the fact that $(m+1)+(m+1)^2+\cdots +(m+1)^{n+1}=\frac{(m+1)^{n+2}-(m+1)}{(m+1)-1}$, we have the result.
\end{proof}

\begin{proposition}\label{prop:lazardstopnumbereven}
Suppose $\ell=2n$ for $n \ge 9$, and let $u=m(m-1)m\cdots m(m-1)$, where $u$ has length $n$, so there are $n-3$ copies of $m$ in the middle. Then, the number of Nyldon words lexicographically greater than $u$ and with length at most $\ell$ is $\displaystyle \frac{(m+1)^n-(m+1)}{(m+1)-1}\ - ((m+1)^4+(m+1)^3+(m+1)^2+(m+1)+3)$.
\end{proposition}
\begin{proof}
As before, we will count the words beginning with $u$ of length at most $\ell$ that are not Nyldon. Let $w$ be such a word. We do casework based on the first Nyldon factor $u$ of $w$. Note that by \cref{prop:firstblock}, performing Melan\c{c}on's algorithm on $w$ yields a block of length $n-2$ at some step, so $|u| \ge n-2$.\\

\noindent\textbf{Case 1}: Suppose $|u|=n-2$. The next Nyldon factor is at least as big as $u$ lexicographically, so starting at index $n-1$, we must have the word $u(m(m-1))^{-1}$. By \cref{lem:repunitnyldonword}, the last $0$, $1$, $2$, $3$, or $4$ digits can be anything as long as $u(m(m-1))^{-1}$ ends with at least $5$ copies of $m$, or $n \ge 9$. So, we get $(m+1)^4+(m+1)^3+(m+1)^2+(m+1)+1$ possible words. \\

\noindent\textbf{Case 2}: Suppose $|u|=n-1$. Then the next Nyldon factor starts with the letter $(m-1)$, which is impossible since it would be lexicographically smaller than $u$, giving $0$ possible words. \\

\noindent\textbf{Case 3}: Suppose $|u|=n$. Then, the rest of $w$ must be either $u$ or $u(m-1)^{-1}m$, since it starts with (and therefore is) a Nyldon word lexicographically at least $u$. So, we get $2$ possible words.\\

\noindent\textbf{Case 4}: Suppose $|u|>n$. Then, the rest of $w$ has length at most $n-1$ and is therefore lexicographically smaller than $u$, which is a contradiction, giving $0$ possible words. \\

Summing over the cases yields a total number of non-Nyldon words equal to $((m+1)^4+(m+1)^3+(m+1)^2+(m+1)+1)+(2)=((m+1)^4+(m+1)^3+(m+1)^2+(m+1)+3)$. There are $(m+1)^k$ words of length $n+k$ starting with $u_i$ for $1 \le k \le n$, and all are Nyldon except the aforementioned words. So, using the fact that $(m+1)+(m+1)^2+\cdots +(m+1)^n=\frac{(m+1)^n-(m+1)}{(m+1)-1}$, we have the result.
\end{proof}
In the Lazard procedure of length $\ell$, every Nyldon word of length at most $\ell$ appears as $u_i$ for some $i$. Then we can simply subtract the number of steps from the end that all the words are generated from the total number of Nyldon words to find the step at which the Lazard procedure of length $\ell$ generates all Nyldon words of length $\ell$.
\begin{definition}
For a Lazard procedure of length $n$ on a right Lazard set $F$ with $|F \cap A^{\le n}|=k$, we define the finishing time to be the smallest $j$ such that $$\bigcup_{i=1}^j \left(Y_i \cap A^{\le n}\right)=\bigcup_{i=1}^k \left(Y_i \cap A^{\le n}\right).$$
\end{definition}

For example, the Lazard procedure of length $5$ on binary Nyldon words (shown in Table 2) has finishing time $4$. Using \cref{prop:lazardstopnumberodd} and \cref{prop:lazardstopnumbereven}, we have the following result:
\begin{corollary}\label{cor:finishingtimes}
The finishing time for the Lazard procedure of length $\ell$ on Nyldon words over an alphabet of size $(m+1)$ is 
$$(\emph{number of Nyldon words up to length }\ell)+1-f(\ell),$$
where $f(\ell)$ is a function equal to $$ \frac{(m+1)^{n+2}-(m+1)}{(m+1)-1}\ - ((m+1)^3+(m+1)^2+2(m+1)+2)$$ for $\ell=2n+1, n \ge 7$, and $$ \frac{(m+1)^n-(m+1)}{(m+1)-1}\ - ((m+1)^4+(m+1)^3+(m+1)^2+(m+1)+3)$$ for $\ell=2n, n \ge 9$.
\end{corollary}
Note that the $+1$ comes from the fact that the crucial lexicographically maximal word (from \cref{prop:lazardstopword}) is not removed until one step after the $Y_i$ corresponding to that word.

We can find more codes made up of Nyldon words by using Lazard sets.

\begin{theorem}\label{thm:kraft}[Kraft-McMillan Inequality] \cite{kraft, mcmillan}
For a code $\{s_1, s_2, \dots\}$ over an alphabet of size $r$, we have
$$ \sum_i r^{-|s_i|} \le 1.$$
\end{theorem}

\begin{theorem}\label{thm:lazardkraft}
Fix a positive integer $n$. For $F$ the set of Nyldon words and $F \cap A^{\le n} = \{ u_1, u_2, \dots, u_k\}$ with $u_1 <_{\text{lex}} u_2 <_{\text{lex}} \cdots <_{\text{lex}} u_k$, let $Y_i$ be the sequence of Lazard sets defined earlier. Note that each $Y_i$ besides $Y_1$ is infinite. Then, each $Y_i$ forms a code and satisfies the equality case of the Kraft-McMillan inequality.
\end{theorem}
\begin{proof}
We use induction. The base case $Y_1=\{0, 1, \dots, m\}$ is clearly a code satisfying the equality case of the Kraft-McMillan equality. Suppose that $Y_i=\{a_1, a_2, \cdots \}$ satisfies these conditions, and $Y_{i+1}=(Y_i \backslash a_j)a_j^*$ for some $j$, where $a_j=u_i$ in the definition of a right Lazard set. Every word $w$ that can be formed from the words in $Y_i$ can clearly be formed in at most one way from the words in $Y_{i+1}$, obtained by putting the non-$a_j$ words forming $w$ together, and adding the appropriate amount of $a_j$'s in between. So, $Y_{i+1}$ is a code.
Now, we need to show that $Y_{i+1}$ satisfies the equality case of the Kraft-McMillan inequality. Since $Y_i$ satisfies the equality case of the Kraft-McMillan inequality, we have
$$
\sum_{a_k \in Y_i} \frac{1}{(m+1)^{|a_k|}}=1.
$$

By deleting $a_j$, we get a geometric series sum of \begin{align*} &\sum_{\substack{a_k \in Y_i \\ a_k \neq a_j}} \frac{1}{(m+1)^{|a_k|}} \cdot \left( 1+\frac{1}{(m+1)^{|a_j|}} + \frac{1}{(m+1)^{2|a_j|}} + \cdots \right) \\=& \sum_{\substack{a_k \in Y_i \\ a_k \neq a_j}} \frac{1}{(m+1)^{|a_k|}} \Bigg/ \left(1-\frac{1}{(m+1)^{|a_j|}}\right)\\=&\frac{1-\frac{1}{(m+1)^{|a_j|}}}{1-\frac{1}{(m+1)^{|a_j|}}}\\=&1.
\end{align*}
\end{proof}

Notably, the proof of \cref{thm:lazardkraft} does not depend on the choice of $u_i$ in each step of the Lazard procedure.

\section{Lyndon Words}
Recall that Lyndon words are those words $w$ that cannot be factorized into a sequence of smaller Lyndon words $w_1w_2\cdots w_k$ with $w_1 \ge_{\text{lex}} w_2 \ge_{\text{lex}} \cdots \ge_{\text{lex}} w_k$. We give a positive answer to a question posed in \cite{nyldon}.

\begin{theorem}\label{thm:lyndonstatement}
If a word $w$ is lexicographically smaller than all of its Lyndon proper suffixes, then it is Lyndon.
\end{theorem}
\begin{proof}
We will only use the recursive definition of Lyndon words to prove this implication.
Suppose $w$ is not Lyndon. Then, by the recursive definition of Lyndon words, $w$ has a factorization $w_1w_2\cdots w_k$ into Lyndon words with $w_1 \ge_{\text{lex}} w_2 \ge_{\text{lex}} \cdots  \ge_{\text{lex}} w_k$ and $k \ge 2$. But then $w_k$ is a Lyndon proper suffix of $w$ and $w >_{\text{lex}} w_1 \ge_{\text{lex}} w_k$, so $w$ is not lexicographically smaller than all of its Lyndon proper suffixes, and we are done.
\end{proof}

\section{Further Directions}
In \cite{nyldon}, Charlier et al. give the problem of describing the forbidden prefixes of Nyldon words (Open Problem 11). They also ask whether prefixes of Nyldon words must always be sesquipowers (or fractional powers) of Nyldon words (Open Problem 12). These problems both remain open, and we in general know very little about prefixes of Nyldon words.

We also wonder how much the bound $K \le \lfloor \log_2(|w|)\rfloor +1$ can be improved in Nyldon factorizing $w^k$.
\begin{question}
For a primitive word $w$ with Nyldon conjugate $n$, and large enough $k$ depending on $w$, what is the best bound $K$ such that the word $w^k$ can be factorized as $(p_1, p_2, \dots, p_i, n, n, \dots, n,$ $ s_1, s_2, \dots, s_j),$ where there are $k-K$ copies of $n$ in the middle? In particular, does a constant bound on $K$ exist? What is the best bound $K$ for a general Nyldon-like set?
\end{question}

So far, we have not found any word $w$ with a value of $K$ more than $4$ in the case of Nyldon words. Note that $w=01111011011111011110111$ has $K=4$ in \cref{exa:infection}. It may be interesting to find the Nyldon-like set with the highest asymptotic value of $K$. Perhaps a bound can also be found for other kinds of Hall sets, such as Viennot sets.
\begin{question}
Fix a Viennot set $(V, \prec)$. For a primitive word $w$ with conjugate $n$ in $V$, and large enough $k$ depending on $w$, does there exist a $K$ such that the word $w^k$ can be factorized as $(p_1, p_2, \dots, p_i, $ $n, n, \dots, n, s_1, s_2, \dots, s_j)$ for $k\ge K$ where there are $k-K$ copies of $n$ in the middle?
\end{question}
Note that $K=1$ for Lyndon words \cite{nyldon}, the prototypical example of a Viennot set. Finally, we ask about the Lazard procedure.
\begin{question}
What is the general distribution on ``finishing times'' for Lazard procedures that generate Hall sets?
\end{question}

\section{Acknowledgments}
This research was funded by NSF/DMS grant 1659047 and NSA grant H98230-18-1-0010. The author would like to thank Prof.\ Joe Gallian for organizing the Duluth REU where this research took place, as well as advisors Aaron Berger and Colin Defant. The author would also like to thank Amanda Burcroff and Sumun Iyer for providing helpful comments on drafts of this paper. The author would also like to thank Spencer Compton for pointing the author toward helpful resources discussing the complexity of certain algorithms. Finally, the author would like to thank the reviewers for several helpful comments.

\bibliographystyle{abbrv}
\bibliography{biblio}

\end{document}